\documentclass[11pt]{amsart}

\usepackage{amsmath, amsthm, amssymb, amscd, amsfonts, enumerate}
\usepackage{fontsmpl}
\usepackage{fontsmpl,layout}
\usepackage[all]{xy}
\usepackage[dvips, dvipsnames, usenames]{color}

\usepackage[latin1]{inputenc}
\usepackage{hyperref}

\newtheorem{lema}{Lemma}[section]
\newtheorem{theorem}[lema]{Theorem}

\newtheorem{prop}[lema]{Proposition}

\newtheorem{af}[lema]{\it\underline{Claim}}

\theoremstyle{definition}
\newtheorem{definition}[lema]{Definition}

\theoremstyle{remark}
\newtheorem{obs}[lema]{Remark}
\newtheorem{rmk}[lema]{Remarks}

\theoremstyle{plain}

\newcommand{\nc}{\newcommand}

\nc{\im}{\mathtt{i}}

\nc{\RR}{{\Bbb R}} \nc{\CC}{{\Bbb C}} \nc{\ZZ}{{\Bbb Z}}
\nc{\FF}{{\Bbb F}} \nc{\NN}{{\Bbb N}} \nc{\QQ}{{\Bbb Q}}
\nc{\PP}{{\Bbb P}} \nc{\KK}{{\Bbb K}} \nc{\DD}{{\Bbb D}}

\nc{\vs}{\vspace{.5cm}}

\nc{\no}{\smallbreak\noindent}

\newcommand\id{\operatorname{id}}
\newcommand\ad{\operatorname{ad}}

\newcommand\co{\operatorname{co}}

\newcommand\ord{\operatorname{ord}}

\newcommand\End{\operatorname{End}}

\newcommand\Rad{\operatorname{Rad}}
\newcommand\Rep{\operatorname{Rep}}

\newcommand\Comod{\operatorname{Comod}}

\newcommand{\eps}{\varepsilon}
\newcommand{\ot}{\otimes}

 \nc{\D}{\Delta} \nc{\e}{\varepsilon}
\nc{\GL}{\operatorname{GL}} \nc{\wact}{\rightharpoonup}
\nc{\Tr}{\mathrm{Tr}} \nc{\cark}{char\,k} \nc{\adl}{\ad_\ell}

\nc{\cP}{\mathcal{P}} \nc{\cU}{\mathcal{U}} \nc{\fD}{\mathfrak{D}}
\nc{\cE}{\mathcal{E}} \nc{\cS}{\mathcal{S}}

\nc{\Ho}{H_0} \nc{\GH}{G(H)} \nc{\mas}{\oplus}
\nc{\coM}{\mathcal{M}^\ast(2,k)} \nc{\cA}{\mathcal{A}}
\nc{\yd}{^{C_2}_{C_2}\mathcal{YD}} \nc{\PH}{\cP(H)}
\nc{\Ftwist}{\overset{\curvearrowright}F}
\nc{\Dtwist}{\overset{\curvearrowright}D}

\begin{document}
\title[Hopf algebras of dimension 16]
{Hopf algebras of dimension 16}

\author[g. a. garc\'\i a and c. vay]
{Gast\' on Andr\' es Garc\'\i a and Cristian Vay}

\address{FaMAF-CIEM (CONICET), Universidad Nacional de C\'ordoba
Medina Allende s/n, Ciudad Universitaria, 5000 C\' ordoba, Rep\'
ublica Argentina.} \email{ggarcia@mate.uncor.edu,
vay@mate.uncor.edu}

\thanks{The work was partially supported by CONICET,
FONCyT-ANPCyT, Secyt (UNC), Agencia C\'ordoba Ciencia.}

\date{\today}

\begin{abstract}
We complete the classification of Hopf algebras of dimension 16
over an algebraically closed field of characteristic zero. We show
that a non-semisimple Hopf algebra of dimension 16, has either the
Chevalley property or its dual is pointed.
\end{abstract}

\maketitle



\section{Introduction}

Let $k$ be an algebraically closed field of characteristic $0$. In
1975, I. Kaplansky posed the question of classifying all Hopf
algebras over $k$ of a fixed dimension. Since the only semisimple
and pointed Hopf algebras are the group algebras, we shall adopt
the convention that `pointed' means `pointed non-semisimple'. Many
results have been found, dealing mainly with the semisimple or
pointed cases.

\smallbreak Concerning Kaplansky´s question, there are very few
general results. The Kac-Zhu Theorem \cite{Z}, states that a Hopf
algebra of prime dimension is isomorphic to a group algebra. S-H. Ng
\cite{Ng} proved that in dimension $p^{2}$, the only Hopf algebras
are the group algebras and the Taft algebras, using previous results
in \cite{andrussch}, \cite{masuoka3}. It is a common belief that a
Hopf algebra of dimension $pq$, where $p$ and $q$ are distinct prime
numbers, is semisimple. Hence, it should be isomorphic to a group
algebra or a dual group algebra by \cite{gelakiwest},
\cite{etinofgelaki}, \cite{masuoka2}. This conjecture was verified
for some particular values of $p$ and $q$, see \cite{andrunatale,
bitidasca, etinofgelaki2, Ng2, Ng3}. In particular, it is known that
Hopf algebras of dimension $14$ and $15$ are group algebras or dual
group algebras \cite{bitidasca}, \cite{andrunatale}.

\smallbreak In fact, all Hopf algebras of dimension $\leq 15$ are
classified: for dimension $\leq 11$ the problem was solved by
\cite{W}; an alternative proof appears in \cite{stefan}. The
classification for dimension 12 was done by \cite{F} in the
semisimple case and then completed by \cite{natale} in the general
case.

\smallbreak It turns out that any Hopf algebra of dimension $\leq
15$ is either semisimple or pointed or its dual is pointed. On the
other hand, there exist Hopf algebras of dimension 16 that are
non-semisimple, non-pointed and their duals are also non-pointed.
Nevertheless, these Hopf algebras satisfies a certain property which
we call the \textit{Chevalley property}.

\smallbreak Recall that a tensor category $\mathcal{C}$  over $k$
has the Chevalley property if the tensor product of any two simple
objects is semisimple. We shall say that a Hopf algebra $H$ has
the \emph{Chevalley property} if the category $\Comod (H)$ of
$H$-comodules does.

\begin{rmk}
(i) The notion of the Chevalley property in the setting of Hopf
algebras was introduced by \cite{AEG}: it is said in \textit{loc.
cit.} that a Hopf algebra has the Chevalley property if the category
$\Rep(H)$ of $H$-modules does.

\smallbreak (ii) Unlike \cite{AEG}, in \cite[Section
1]{de1tipo6chevalley}, the authors refer the Chevalley property to
the category of $H$-comodules; this definition is the one we
adopt. Note that it is equivalent to say that the coradical of $H$
is a Hopf subalgebra.

\smallbreak (iii) If $H$ is semisimple or pointed then it has the
Chevalley property.
\end{rmk}

Here is the main result of the present paper.

\begin{theorem}\label{main result}
Let $H$ be a Hopf algebra of dimension 16. If $H$ does not have
the Chevalley property then $H^*$ is pointed.
\end{theorem}

As a consequence of Theorem \ref{main result}, we obtain the
classification of Hopf algebras of dimension 16.

\begin{theorem}\label{main result2}
Let $H$ be a Hopf algebra of dimension 16. Then $H$ is isomorphic
to one and only one of the Hopf algebras in the following list.
\begin{enumerate}[(i)]
\item The group algebras of groups of order $16$ and their duals.

\item The semisimple Hopf algebras listed in \cite[Thm.
1.2]{kashina}.

\item The pointed Hopf algebras listed in \cite[Section
2.5]{pointed16}.

\item  The duals of the pointed Hopf algebras listed in \cite[Sec.
4.2, Table 2]{biti}.

\item The two non-semisimple non-pointed self-dual Hopf algebras
with the Chevalley property listed in \cite[Thm.
5.1]{de1tipo6chevalley}.
\end{enumerate}
\end{theorem}

\begin{proof}
Let $H$ be a Hopf algebra of dimension 16. If $H$ is semisimple,
then $H$ is either a group algebra, or a dual of a group algebra or
is one of the list given in \cite[Thm. 1.2]{kashina}. Suppose now
that $H$ is non-semisimple. If it is pointed, then $H$ is one of the
Hopf algebras given in \cite[Section 2.5]{pointed16}. If $H$ is
non-pointed and has the Chevalley property, then it must be one of
the two Hopf algebras given by \cite[Thm. 5.1]{de1tipo6chevalley}.
Then, the result follows from Theorem \ref{main result}.
\end{proof}

\smallbreak The paper is organized as follows. In Section
\ref{extensions} we recall the definition and some known facts
about Hopf algebra extensions. We give a detailed description of
the cleft extensions of the Sweedler algebra $T_{-1}$ in
Subsection \ref{Cleft extensions of the Sweedler algebra}, a
result from \cite{masuoka}-- we follow the exposition in
\cite{doitake}. As a consequence we show in Lemma \ref{lo que
creia que era solo el tensor}, that a Hopf algebra which is an
extension of $T_{-1}$ by $T_{-1}$ is isomorphic to the tensor
product $T_{-1}\ot T_{-1}$. In Section \ref{la clasificacion de
stefan} we recall the classification of non-semisimple Hopf
algebras of dimension $8$ given by \cite{stefan}, since it is used
several times in the proof of our main theorem. In Section
\ref{Hopf algebras generated by simple coalgebras} we discuss some
consequences of results of \cite{natale} and \cite{stefan}
concerning Hopf algebras $H$ generated by a simple subcoalgebra of
dimension $4$ stable by the antipode. In particular, we show in
Theorem \ref{2n tiene dual punteado} that under certain
assumptions $H^{*}$ must be pointed. Finally we prove our main
theorem in Section \ref{Proof of the Main Theorem}. We first
describe all possible coradicals of a Hopf algebra of dimension 16
which does not have the Chevalley property. It turns out that
there are $6$ possible coradicals. This leads us to do the proof
case by case according to the type of the coradical. The most
difficult cases are those where the coradical has two simple
subcoalgebras $C$ and $D$ of dimension $4$, since one does not
know whether they are stable by the antipode. The problem is
solved by looking at the subalgebra generated by $C$, which is
indeed a Hopf subalgebra, in the case that both $C$ and $D$ are
stable by the antipode. In the other case, one assumes that
$H^{*}$ is non-pointed and then one gets a contradiction by
looking at the Hopf subalgebras of dimension $8$ contained in it.

\smallbreak If $H$ is a Hopf algebra over $k$ then $\D$, $\e$, $S$
denote respectively the comultiplication, the counit and the
antipode; $\GH$ denotes the group of group-like elements of $H$;
$(H_n)_{n\in\NN}$ denotes the coradical filtration of $H$; $L_h$
(resp. $R_h$) is the left (resp. right) multiplication in $H$ by
$h$. The left and right adjoint action $\adl, \ad_r:H\rightarrow
\End(H)$, of $H$ on itself are given, in Sweedler notation, by:
$$
\adl(h)(x) = \sum h_1xS(h_2), \qquad \ad_r(h)(x) = \sum
S(h_1)xh_2,
$$
for all $h,x\in H$. The set of $(g,h)${\it -primitives} (with
$h,g\in\GH$) and {\it skew-primitives} are:
$$
\begin{array}{rcl}
\cP_{g,h}(H)&:=&\{x\in H\mid\D(x)=x\ot h+g\ot x\},\\
\noalign{\smallskip} \cP(H)&:=&\sum_{h,g\in\GH}\cP_{h,g}(H).
\end{array}
$$
We say that $x\in k\cdot(h-g)$ is a {\it trivial} skew-primitive;
otherwise, it is {\it non-trivial}.

Let $K$ be a coalgebra with a distinguished group-like 1. If $M$
is a right $K$-comodule via $\delta$, then the space of {\it right
coinvariants} is $$ M^{\co \delta} = \{x\in
M\mid\delta(x)=x\ot1\}.
$$ In particular, if $\pi:H\rightarrow K$ is a morphism of Hopf
algebras, then $H$ is a right $K$-comodule via $(1\ot\pi)\D$ and
in this case $H^{\co \pi}:=H^{\co (1\ot\pi)\D}$.

Let ${\mathcal M}^\ast(n,k)$ denote the simple coalgebra of
dimension $n^2$, dual to the matrix algebra ${\mathcal M}(n,k)$. A
basis $\{e_{ij}\mid1\leq i,j\leq n\}$ of ${\mathcal M}^\ast(n,k)$
such that $\D(e_{ij})=\sum_{l=1}^ne_{il}\ot e_{lj}$ and
$\e(e_{ij})=\delta_{ij}$ is called a {\it comatrix basis.}

\section{Extensions}\label{extensions}
Our references for the theory of extensions of algebras and Hopf
algebras are \cite{andrudevoto} and \cite{mongomeri}.

\subsection{Extensions of Hopf algebras}\label{ext de Hopf}

\begin{definition}
Let $A\subset C$ be an extension of $k$-algebras and $B$ be a Hopf
algebra. $A\subset C$ is a {\it $B$-cleft extension} if $C$ is a
(right) $B$-comodule algebra via $\delta$ with $C^{\co \delta}=A$
and there is  $\gamma:B\rightarrow C$ a morphism of $B$-comodules
which is convolution invertible.
\end{definition}

It is known that any cleft extension arises as a {\it crossed
product} $A\#_{\rightharpoonup,\sigma}B$, and conversely any
crossed product is a cleft extension \cite[Thm. 7.2.2]{mongomeri}.
Here $\rightharpoonup:B\ot A\rightarrow A$ is a {\it weak action}
and $\sigma:B\ot B\rightarrow A$ is a {\it $2$-cocycle} satisfying
certain compatibility  conditions, so that $A\ot B$ becomes an
associative algebra with a new product and unit $1\ot1$. The
multiplication is given by:
\begin{equation}\label{producto en el smash}
(a\# b)(a'\# b')=a(b_1\rightharpoonup a')\sigma(b_2,b'_1)\#
b_3b'_2,
\end{equation}
for all $a,a'\in A$ and $b,b'\in B$. See \cite[Section
2]{andrudevoto} or \cite[Section 7]{mongomeri} for details.

\begin{definition}\cite{andrudevoto}.
Let $A\overset{\imath}\hookrightarrow
C\overset{\pi}\twoheadrightarrow B$ be a sequence of Hopf algebras
morphisms. We shall say that it is {\it exact} and $C$ is an {\it
extension of $A$ by $B$} if:
\begin{enumerate}[(i)]
\item $\imath$ is injective (then we identify $A$ with its image);
\item $\pi$ is surjective;
\item $\pi\imath=\e$;
\item $\ker\pi=A^+C$ ($A^+$ is the kernel of the counit);
\item $A=\,C^{\co \pi}$.
\end{enumerate}
\end{definition}

The following statement condenses some known results and is useful
to find exact sequences.

\begin{lema}\label{resumen sobre ext nec}
Let $C$ be a finite-dimensional Hopf algebra. If $\pi:C\rightarrow
B$ is an epimorphism of Hopf algebras then $\dim C=\dim C^{\co
\pi}\dim B$. Moreover, if $A=C^{\co \pi}$ is a Hopf subalgebra
then the sequence $ A\overset{\imath}\hookrightarrow
C\overset{\pi}\twoheadrightarrow B $ is exact.
\end{lema}

\begin{proof}
The equality of dimension follows from \cite[Thm. 2.4.
(1.b)]{sch}. Moreover, if $A=C^{\co \pi}$ then $\pi_{|A}=\e_{|A}$
and therefore $A^+C\subseteq\ker\pi$. It follows from \cite[Thm.
2.4. (2.a)]{sch} that $\dim B=\dim(C/A^+C)$. Therefore
$A^+C=\ker\pi$, and the lemma follows.
\end{proof}

Exact sequences of finite-dimensional Hopf algebras are cleft by
\cite[Thm. 2.2]{sch} so by the results in \cite[Subsection
3.2]{andrudevoto} we have the following. Recall the definition of
Hopf datum \cite[Def. 2.26]{andrudevoto} and the corresponding
Hopf algebra $A\,^{\rho,\tau}\#_{\rightharpoonup,\sigma}B $
associated to it.

\begin{theorem}\label{recuperar el centro de la ext} Let $A$ and
$B$ be finite-dimensional Hopf algebras.
\begin{enumerate}[(i)]
\item Let $ A\overset{\imath}\hookrightarrow
C\overset{\pi}\twoheadrightarrow B$ be an exact sequence of Hopf
algebras. Then $C$ is finite-dimensional and there exists a Hopf
datum $(\rightharpoonup,\sigma,\rho,\tau)$ such that $C\simeq
A\,^{\rho,\tau}\#_{\rightharpoonup,\sigma}B$ as Hopf algebras.

\item Conversely, if $(\rightharpoonup,\sigma,\rho,\tau)$ is a Hopf
datum over $A$ and $B$, then the maps $\imath(a)=a\#1$ and
$\pi(a\# b)=\e(a)b$ are morphisms of Hopf algebras and give rise
to an exact sequence of Hopf algebras $$
A\overset{\imath}\hookrightarrow
A\,^{\rho,\tau}\#_{\rightharpoonup,\sigma}B\overset{\pi}\twoheadrightarrow
B.$$

\item\label{como tiene q ser el morfismo para tener extension de hopf eq}
Let $\phi:B\rightarrow A$ be a convolution-invertible linear map
such that $\phi(1)=1$ and $\e\circ\phi=\e$. Then $
A\,^{\rho,\tau}\#_{\rightharpoonup,\sigma}B\simeq
A\,^{\rho^{\phi^{-1}},\tau^{\phi^{-1}}}\#_{\,^\phi\rightharpoonup,\,^\phi\sigma}B
$ for any Hopf datum $(\rightharpoonup,\sigma,\rho,\tau)$.
$\qquad\square$
\end{enumerate}
\end{theorem}

In particular, the last part of the Theorem says that
$A\#_{\,^\phi\rightharpoonup,\,^\phi\sigma}B\simeq
A\#_{\rightharpoonup,\sigma}B$ as cleft extensions.

\subsection{Cleft extensions of the Sweedler algebra
$T_{-1}$}\label{Cleft extensions of the Sweedler algebra}

Given an algebra $A$ and a Hopf algebra $B$, in general, it is not
easy to find a compatible pair $(\rightharpoonup,\sigma)$ giving
rise to a crossed product $A \#_{\rightharpoonup,\sigma}B$.
However, the classification given by \cite{doitake} and
\cite{masuoka} provides a way to construct all compatible pairs
$(\rightharpoonup,\sigma)$ when $B=T_{-1}$, the Sweedler algebra
of dimension $4$. Explicitly,
\begin{equation}
\begin{array}{c}
T_{-1}=k\langle g,x\mid g²=1, x²=0, xg=-gx\rangle,\\
\noalign{\smallskip}
\D(g)=g\ot g\quad\mbox{and}\quad\D(x)=x\ot g+1\ot x.
\end{array}
\end{equation}

\smallbreak

\begin{definition}\cite[Def. 2.4]{doitake}, \cite[Def. 3.1]{masuoka}.
\label{cleft datos}
Let $A$ be an algebra. A $5$-tuple
$\fD=(F,D,\alpha,\beta,\gamma)$, where $F,D\in\End(A)$,
$\alpha\in\cU(A)$ (the units of $A$) and $\beta,\gamma\in A$ is
called a {\it$T_{-1}$-cleft datum over $A$} if it satisfies:
$$
\begin{array}{l}
\begin{array}{ll}
(\fD1)\,F\mbox{ is an algebra morphism,}&(\fD2)\,D(aa')=aD(a')+D(a)F(a'),\\
\noalign{\smallskip}
(\fD3)\,F^2(a)\alpha=\alpha,&\hspace{-15pt}(\fD4)\,(FD(a)+DF(a))\alpha=\gamma
a-F(a)\gamma,
\end{array}
\\ \noalign{\smallskip}
\begin{array}{ll}
(\fD5)\,D(a)\gamma+D^2(a)\alpha=\beta a-a\beta,&(\fD6)\,F(\alpha)=\alpha,
\qquad(\fD7)\,D(\beta)=0,\\
\noalign{\smallskip}
(\fD8)\,D(\alpha)=\gamma-F(\gamma),&(\fD9)\,D(\gamma)=\beta-F(\beta),
\end{array}
\end{array}
$$
for all $a,a'\in A$.
\end{definition}

\begin{definition}\label{algebra de la cleft datos}
\cite[Thm. 2.3, Def. 2.4]{doitake}, \cite[Prop. 3.4]{masuoka}. If
$\fD=(F,D,\alpha,\beta,\gamma)$ is a $T_{-1}$-cleft datum over
$A$, then $C_{\fD}:=A\#_{\rightharpoonup,\sigma} T_{-1}$ is an
associative algebra where $\rightharpoonup:T_{-1}\ot A\rightarrow
A$ is the weak action given by:
$$
1\rightharpoonup a=a,\,g\rightharpoonup a=F(a),\,x\rightharpoonup a=D(a),
\,(gx)\rightharpoonup a=FD(a)\alpha,
$$
and $\sigma:T_{-1}\ot T_{-1}\rightarrow T_{-1}$ is the $2$-cocycle
given by the following table:

\bigbreak
\begin{center}
\begin{tabular}{|c|c|c|c|c|}
\hline $\sigma$ & $1$ & $g$ & $x$ & $gx$ \\
\hline $1$ & $1$ & $1$ & $0$ & $0$ \\
\hline $g$ & $1$ & $\alpha$ & $0$ & $0$ \\
\hline $x$ & $0$ & $\gamma$ & $\beta$ & $-F(\beta)$ \\
\hline $gx$ & $0$ & $F(\gamma)$ & $F(\beta)$ & $-\alpha\beta$\\
\hline
\end{tabular}
\end{center}
\end{definition}

\smallbreak

The $T_{-1}$-cleft data classify all $T_{-1}$-cleft extensions:

\begin{theorem}\label{isomorfismos de las cleft datos}
\cite[Cor. 2.5, Thm. 2.7]{doitake}, \cite[Prop. 3.4]{masuoka}.
\begin{enumerate}[(i)]
\item If $A\subset C$ is a $T_{-1}$-cleft extension, then it is
isomorphic to $C_{\fD}$ for some $T_{-1}$-cleft datum $\fD$ over
$A$. \item Let $\fD=(F,D,\alpha,\beta,\gamma)$ and
$\fD'=(F',D',\alpha',\beta',\gamma')$ be $T_{-1}$-cleft data over
an algebra $A$. Then $C_{\fD}\simeq C_{\fD'}$ as
$T_{-1}$-extensions if only if there exists element $s\in\cU(A)$
and $t\in A$ such that for all $a\in A$:
$$
\begin{array}{l}
\begin{array}{ll}
(C_\fD1)\quad F'(a)=sF(a)s^{-1},&(C_\fD2)\quad D'(a)=(tF(a)+D(a)-at)s^{-1},\\
\noalign{\smallskip}
(C_\fD3)\quad\alpha'=sF(s)\alpha,&(C_\fD4)\quad\beta'=\beta+t\gamma+(tF(t)+D(t))\alpha,
\end{array}
\\ \noalign{\smallskip}
\begin{array}{l}
(C_\fD5)\quad\gamma'=s\gamma+(tF(s)+D(s)+sF(t))\alpha.
\qquad\square
\end{array}
\end{array}
$$
\end{enumerate}
\end{theorem}

Moreover, there is a linear map $\phi:T_{-1}\rightarrow A$ such
that $(^\phi\rightharpoonup,\,^\phi\sigma)=$
$(\rightharpoonup',\sigma')$, see for example \cite[Prop.
3.2.12]{andrudevoto}. It is given explicitly by:
\begin{equation}\label{como es el morfismo que conecta a dos extensiones}
\phi(1)=1, \quad \phi(g)=s,\quad  \phi(x)=t \quad \mbox{ and }
\quad \phi(gx)=sF(t)\alpha.
\end{equation}

\smallbreak Next we list some properties of $T_{-1}$ that will be
useful in the sequel.
\begin{enumerate}
\item[$(T1)$] $\Rad T_{-1}=k\cdot x\oplus k\cdot gx$\, and\,
$hh'=0\ \forall\ h,h'\in\Rad T_{-1}$. \smallbreak \item[$(T2)$]
$\cU(T_{-1})=\{a+bg+h\mid a,b\in k,\,a²-b²\neq0,\,h\in\Rad
T_{-1}\}$ (multiply by $a-bg+h$ and use that $h^2=0$). \smallbreak
\item[$(T3)$] $\{t\in T_{-1}\mid t²=1\}=\{\pm1,\pm g+h\mid
h\in\Rad T_{-1}\}$. \smallbreak \item[$(T4)$] $\forall\ h\in
k[G(T_{-1})]$\, there exists\, $s\in k[G(T_{-1})]$ such that
$s^2=h$ (write the necessary equations to find $s$ and solve
them-- that is possible because $k$ is an algebraically closed
field of characteristic zero).
\end{enumerate}

We end this section by proving a theorem which determines all
possible extensions of $T_{-1}$ by $T_{-1}$ (up to isomorphisms).

\begin{lema}\label{lo que creia que era solo el tensor}
If $T_{-1}\overset{\imath}\hookrightarrow H\overset{\pi}\twoheadrightarrow T_{-1}$
is an exact sequence of Hopf algebras
then $H\simeq T_{-1}\ot T_{-1}$.
\end{lema}

\begin{proof}
By \ref{recuperar el centro de la ext}, $H\simeq
T_{-1}\,^{\rho,\tau}\#_{\rightharpoonup,\sigma}T_{-1}$ for some
Hopf datum $(\rightharpoonup,\sigma,\rho,\tau)$. In particular,
$T_{-1}\subset H$ is a $T_{-1}$-cleft extension. So
$T_{-1}\#_{\rightharpoonup,\sigma}T_{-1}\simeq C_{\fD}$ as
algebras, where $\fD$ is a $T_{-1}$-cleft datum over $T_{-1}$.

\smallbreak Our aim is to change the initial $T_{-1}$-cleft datum
$\fD$ by another equivalent but more appropriate, using
\ref{isomorfismos de las cleft datos}, in such a way that we still
have an exact sequence of Hopf algebras. Because of \eqref{como es
el morfismo que conecta a dos extensiones} and \ref{recuperar el
centro de la ext} (\ref{como tiene q ser el morfismo para tener
extension de hopf eq}), this is possible if the following
conditions for $s \in \mathcal{U}(T_{-1})$ and $t\in T_{-1}$ are
satisfied:
\begin{equation}\label{aun es exacta}
\e(s)=1,\qquad \e(t)=0\quad \mbox{ and }\quad\e(F(t))=0.
\end{equation}

Let $\fD=(F,D,\alpha,\beta,\gamma)$ be a $T_{-1}$-cleft datum over
$T_{-1}$. By $(\fD1)$ and $(\fD3)$, $F$ is an algebra automorphism
of $T_{-1}$. Then by (T3), $F(g)=\pm g+h$ for some $h\in\Rad
T_{-1}$. Actually, $F(g)=g+h$. In fact,
\begin{eqnarray}\label{g por g}
(1\# g)(g\# 1)=(g\rightharpoonup g)\sigma(g,1)\ot g=F(g)\# g,
\end{eqnarray}
the last equality follows from \ref{algebra de la cleft datos}. If
we apply $\pi$, the Hopf algebra morphism defined in
\ref{recuperar el centro de la ext}, we find that
$\e(F(g))=\e(g)$. Therefore $F(g)=g + h$.

Let $s=g+\frac{h}{2}$, $t=0$ and $\phi:T_{-1}\rightarrow A$ as in
\eqref{como es el morfismo que conecta a dos extensiones}. Then
the algebra automorphism $F'$ corresponding to the new cleft datum
$\fD'$ equivalent to $\fD$  satisfies
\begin{equation}\label{buscando a sigma 3}
F'(g)= g
\end{equation}
by $(C_\fD1)$; and we still have an exact sequence of Hopf
algebras by  \eqref{aun es exacta}. For simplicity, we still write
$\fD$ for $\fD'$.

\smallbreak We now perform a second change of datum. By $(\fD3)$
with $a=g$, we have that $\alpha\in k[G(T_{-1})]$. Moreover,
$\e(\alpha)=1$ since
\begin{eqnarray}\label{g por g bis}
(1\# g)(1\# g)=(g\rightharpoonup 1)\sigma(g,g)\ot 1=\alpha\# 1,
\end{eqnarray}
the last equality by \ref{algebra de la cleft datos}. Applying
$\pi$, it follows that $\e(\alpha)=1$. By $(T4)$, we can pick
$s\in k[G(T_{-1})]$ such that $s²=\alpha^{-1}$; note that
$F(s)=s$. Moreover, we may assume that $\e(s)=1$ since
$(-s)²=\alpha^{-1}$. Let also $t=0$  and $\phi:T_{-1}\rightarrow
A$ as in \eqref{como es el morfismo que conecta a dos
extensiones}. Then the new cleft datum given as in
\ref{isomorfismos de las cleft datos} (ii) has
\begin{equation}\label{buscando a sigma 4}
\alpha'=1,\,F'(g)=g,
\end{equation}
by $(C_\fD1)$ and $(C_\fD3)$; and by \eqref{aun es exacta}, we
still have an exact sequence of Hopf algebras. Again, we write
$\fD$ instead of $\fD'$.

\smallbreak We now perform a further change of datum. Let $s= 1$,
$t=\frac{g}{2}D(g)$ and $\phi:T_{-1}\rightarrow A$ as in
\eqref{como es el morfismo que conecta a dos extensiones}. By
$(\fD2)$, $D(1)=0$ and therefore $0=gD(g)+D(g)g$ also by $(\fD2)$.
Then $D(g)\in\Rad T_{-1}$. Thus, using $(C_{\fD}2)$, the new cleft
datum defined as in \ref{isomorfismos de las cleft datos} (ii) has
\begin{equation}\label{buscando a sigma 5}
 D(g)=0,\,F(g)=g,\,\alpha=1
\end{equation}
and we still have an exact sequence of Hopf algebras
(note that $F(t)\in\Rad T_{-1}$ since $t\in\Rad T_{-1}$).

\smallbreak We perform still another change of datum,
corresponding to $s=1$ and $t=-\frac{1}{2}\gamma$. Indeed, note
that  $0=\gamma g-g\gamma$, by $(\fD4)$ with $a=g$ and
$(\ref{buscando a sigma 5})$. Then $\gamma\in k[G(T_{-1})]$. We
claim moreover that $\e(\gamma)=0$. In fact,
\begin{equation}\label{g por x}
\begin{array}{rl}
(1\# g)(1\# x)=&(g\rightharpoonup 1)\sigma(g,x)\ot
1+(g\rightharpoonup 1) \sigma(g,1)\ot gx\\=&\gamma\# 1+1\# gx;
\end{array}
\end{equation}
the last equality follows from \ref{algebra de la cleft datos}.
Applying $\pi$, it follows that $\e(\gamma)=0$. Then the new cleft
datum has
\begin{equation}\label{buscando a sigma 6}
\gamma=0,\,F(g)=g,\,D(g)=0,\,\alpha=1
\end{equation}
by $(C_{\fD}5)$; and we still have an exact sequence of Hopf
algebras-- note that $F(\gamma)=\gamma$.

\smallbreak We perform the last change of datum, corresponding to
$s=g$ and $t=0$. Since $F$ is an algebra morphism, there exists
$a,b\in k$ such that
$$
F_{|\Rad T_{-1}}=\left(\begin{matrix} a&b\\b&a
                       \end{matrix}\right)
$$
on the basis $\{x,gx\}$. As $\alpha=1$, $F²=\id$ by $(\fD3)$. Then
either $a=\pm1$ and $b=0$ or $a=0$ and $b=\pm1$. By $(C_{\fD}1)$,
the new cleft datum has either
\begin{align}\label{buscando a sigma 7}
F&=\id, & D(g)&=0,&\quad \alpha&=1 &\mbox{ and }\gamma & =0
\quad\mbox{ or }
\\\label{buscando a sigma 8}
F(g)&=g,\quad  F(x)=gx,& D(g)&=0, & \alpha  &=1 &\mbox{ and
}\gamma &=0.
\end{align}
In both cases, we still have exact sequences of Hopf algebras.

\smallbreak We next claim that $D=0$ and $\beta = 0$. In
$(\ref{buscando a sigma 7})$, $D=0$ by $(\fD4)$; hence $\beta\in
k$ (the center of $T_{-1}$) by $(\fD5)$. In (\ref{buscando a sigma
8}), $0=xD(x)+D(x)gx$ by $(\fD2)$, hence $D(x)xg=xD(x)$. If we
write $D(x)=c+dg+h$ with $c,d\in k$ and $h\in\Rad T_{-1}$, then
$$
(c+dg)xg=x(c+dg)\Rightarrow d=c=-d.
$$
Therefore $D(x)=h\in\Rad T_{-1}$. Moreover, since $D(g)=0$,
$D(gx)=gD(x)$. Now, since $\alpha=1$, $\gamma=0$ and $F(h)=gh$ for
all $h\in\Rad T_{-1}$, we see from  $(\fD4)$ that
$$
0=FD(x)+DF(x)=F(h)+D(gx) =gh+gD(x)=gh+gh=2gh.
$$
Therefore $D = 0$, and $\beta$ must belong to $k$ too by $(\fD5)$.
In both cases, we see by \ref{algebra de la cleft datos} that
$x\rightharpoonup a=D(a)=0$, $\sigma(x,1)=\sigma(1,x)=x²=0$ and
\begin{equation}\label{x al cuadrado}
(1\# x)(1\# x)=\sigma(x,x)\ot1=\beta\ot1.
\end{equation}
Applying $\pi$, since $\beta\in k$, it follows that $\beta=0$.

\smallbreak We define the algebra morphism $\Ftwist$ by
$\Ftwist(g):=g$  and $\Ftwist(x):=gx$. Then $H$ must be isomorphic
as algebra to $C_{\fD}$ where $\fD$ is one of the following cleft
data:
\begin{align}\label{cleft dato trivial}
\fD_0 &:=(\id,0,1,0,0)\mbox{ or}
\\\label{cleft dato no trivial}
\overset\curvearrowright\fD &:=(\Ftwist,0,1,0,0).
\end{align}

Our next aim is to show that
\begin{align}\label{cota radical} \Rad(T_{-1})\ot T_{-1}+T_{-1}\ot\Rad(T_{-1})\subseteq\Rad H.
\end{align}

If $\fD=\fD_0$, then $H\simeq T_{-1}\ot T_{-1}$ as algebras and
\eqref{cota radical} follows. If
$\fD=\overset{\curvearrowright}\fD$ then $H\simeq
T_{-1}\#_{\rightharpoonup,\sigma}T_{-1}$. Here
$(\rightharpoonup,\sigma)$ is given by
\begin{align}\label{posible accion debil}
1\rightharpoonup h=h,\quad g\rightharpoonup
h&=\left\lbrace\begin{matrix}h&h\in k[G(T_{-1})]\\gh&h\in\Rad
T_{-1}
\end{matrix}\right.,\quad x\rightharpoonup h=gx\rightharpoonup
h=0;
\\\label{posible 2cociclo}
\sigma(h,h')&=\e(h)\e(h'),
\end{align}
for all $h,h'\in T_{-1}$, by \ref{algebra de la cleft datos}. By
explicit calculations, we have that $\Rad T_{-1}\ot T_{-1}$ and
$T_{-1}\ot\Rad T_{-1}$ are nilpotent ideals of $H$. Then
\eqref{cota radical} follows.

Now \eqref{cota radical} implies that $$\dim(H^*)_0\leq\dim
(H/(\Rad(T_{-1})\ot T_{-1}+T_{-1}\ot\Rad(T_{-1}))) = 16-12=4.
$$

Therefore, any simple representation of $H$ is one-dimensional,
i.e., $H^*$ is pointed. Moreover, since $(T_{-1})^*\simeq T_{-1}$,
$H^*$ is also an extension of $T_{-1}$ by $T_{-1}$. Therefore
$(H^*)^*\simeq H$ is pointed too.

Summarizing, $H$ and $H^*$ are pointed,  the groups $G(H)$ and
$G(H^*)$ have order $\leq 4$ and both contain a normal Sweedler
Hopf subalgebra. By inspection in the classification list of
pointed Hopf algebras of dimension 16 given in \cite{pointed16},
we see that $H$ must be isomorphic to $T_{-1}\ot T_{-1}$.
\end{proof}

\section{Non-semisimple Hopf algebras of dimension 8}
\label{la clasificacion de stefan} We shall need the
classification of the non-semisimple Hopf algebras of dimension 8
\cite{stefan}. We give this list, including the defining relations
of the algebra structure and the comultiplication in terms of the
generators. Let $i$ be a primitive 4-root of 1.
$$
\begin{array}{ll}
{\mathcal A}_2:=&k\langle g,x,y\mid
g^2-1=x^2=y^2=gx+xg=gy+yg=xy+yx=0\rangle,
\\ \noalign{\smallskip}
&\D(g)=g\ot g,\quad\D(x)=x\ot g+1\ot x,\quad\D(y)=y\ot g+1\ot y.
\end{array}
$$
$$
\begin{array}{l}
{\mathcal A}'_4:=k\langle g,x\mid g^4-1=x^2=gx+xg=0\rangle,
\\ \noalign{\smallskip}
\hspace{2cm}\D(g)=g\ot g,\quad\D(x)=x\ot g+1\ot x;
\\ \noalign{\vspace{.3cm}}
 {\mathcal A}''_4:=k\langle g,x\mid g^4-1=x^2-g^2+1=gx+xg=0\rangle,
\\ \noalign{\smallskip}
\hspace{2cm}\D(g)=g\ot g,\quad\D(x)=x\ot g+1\ot x;
\\ \noalign{\vspace{.3cm}}
{\mathcal A}'''_{4,i}:k\langle g,x\mid g^4-1=x^2=gx-ixg=0\rangle,
\\ \noalign{\smallskip}
\hspace{2cm}\D(g)=g\ot g,\quad\D(x)=x\ot g^2+1\ot x;
\\ \noalign{\vspace{.3cm}}
{\mathcal A}_{2,2}:=k\langle g,h,x\mid g^2=h^2=1, \,
x^2=gx+xg=hx+xh=gh-hg=0\rangle,
\\ \noalign{\smallskip}
\hspace{2cm}\D(g)=g\ot g,\quad\D(h)=h\ot h,\quad\D(x)=x\ot g+1\ot
x.
\end{array}
$$

\smallbreak

\begin{rmk}\label{para cuando necesite una sweedler}
There are the following isomorphisms: ${\mathcal
A}_2\simeq({\mathcal A}_2)^*$, ${\mathcal
A}'''_{4,i}\simeq{\mathcal A}'''_{4,-i}\simeq({\mathcal A}'_4)^*$
 and ${\mathcal A}_{2,2}\simeq({\mathcal A}_{2,2})^*$ \cite{stefan}.
Moreover, one can check case-by-case that all these Hopf algebras
have Hopf subalgebras isomorphic to $T_{-1}$.
\end{rmk}

\smallbreak

\subsection{The unique Hopf algebra of dimension 8 which does not have
the Chevalley property} \label{la unica} By \cite{stefan},
$\cA:=({\mathcal A}''_4)^*$ is the unique Hopf algebra of
dimension 8 neither semisimple nor pointed; its coradical is
$\cA_0=k[C_2]\mas\coM$ and $\cA$ is generated as an algebra by
$\coM$.

\smallbreak We next compute explicitly the multiplication of the
elements of a comatrix  basis of $\coM$. For this, we first
describe the simple representations of ${\mathcal A}''_4$. Let $g$
and $x$ be the generators of ${\mathcal A}''_4$.

\begin{lema}\label{las representations de cA dual}
The simple one-dimensional representations  of ${\mathcal A}''_4$
are $\e$ and $\alpha: {\mathcal A}''_4 \longmapsto k$, where
\begin{equation}\label{gr like de A}
\alpha(g) = -1,\quad \alpha(x) = 0.
\end{equation}
The unique (up to isomorphisms) simple representation of dimension
$2$ of ${\mathcal A}''_4$ is
$\rho:\cA^\ast\longmapsto\mathcal{M}(2,k)$,
\begin{equation}\label{comatrixes de A}
\rho(g)=\left(\begin{matrix}i&0\\0&-i\end{matrix}\right),\quad
\rho(x)=\left(\begin{matrix}0&2\\-1&0\end{matrix}\right).
\end{equation}
\end{lema}

\begin{proof}
For simplicity, if $(V,\varrho)$ is simple representation of
$\cA''_4$, we write $a$ instead of $\varrho(a)\in\End V$.

In case that $\dim V=1$, from relations $g^4=1$ and $xg=-gx$ it
follows that $g$ is a 4-root of 1 and $x=0$. Then we have $g^2=1$,
by the relation $x^2-g^2+1=0$. So that either $g=1$ or $g=-1$.
This defines $\eps$ and $\alpha$ respectively.

\smallbreak In case that $\dim V=2$, by $g^4=\id$, we can choose a
basis of $V$ consisting of eigenvectors of $g$. Then the
eigenvalues of $g$ are different 4-roots of 1. In fact, if they
are equal then $x=0$ (by $xg=-gx$) but the representation
$\varrho(g)=i^{j}\cdot\id,\ 0\leq j\leq 3,\ \varrho(x)=0$ is not
simple. Let $\xi$ and $\eta$ be the eigenvalues of $g$. Then
\begin{equation}
xg=\left(\begin{matrix}\xi x_{11}&\eta x_{12}\\
\xi x_{21}&\eta x_{22}\end{matrix}\right)=\left(\begin{matrix}
-\xi x_{11}&-\xi x_{12}\\ -\eta x_{21}&-\eta
x_{22}\end{matrix}\right)=-gx,
\end{equation}
so $x_{11}=x_{22}=0$. Moreover, $x_{12}\neq0\neq x_{21}$. Indeed,
both $x_{12}\neq 0$ and $x_{21}\neq 0$, because the representation
is simple. Therefore $\eta=-\xi$. By $0=x^2-g^2+\id$, we have that
\begin{equation}
0=\left(\begin{matrix}x_{12}x_{21}-\xi^2+1&0\\
0&x_{12}x_{21}-\xi^2+1\end{matrix}\right).
\end{equation}
Since $x_{12}\neq0\neq x_{21}$, it follows that $\xi^2\neq1$.
Therefore $\xi$ is a primitive 4-root of 1 and $x_{12}x_{21}=-2$.
Taking $x_{12}=2$ and $x_{21}=-1$, we find $\rho$.

Since $(\cA''_4)^*_0=\cA_0=k[C_2]\mas\coM$, the lemma follows.
\end{proof}

Let $(k^2,\rho)$ be the 2-dimensional representation given by
\ref{las representations de cA dual}. Let $\{E_{ij}\mid1\leq
i,j\leq2\}$ be the coordinate functions of $\mathcal{M}(2,k)$. If
$e_{ij}:=E_{ij}\circ\rho$, then $\cE_\cA:=\{e_{ij}\mid1\leq
i,j\leq2\}$ is a comatrix basis of the simple subcoalgebra of
$\cA$ isomorphic to $\coM$.

\begin{lema}\label{relaciones de comatrixes de cA}
The elements of $\cE_\cA$ satisfy:
\begin{equation}\label{le1}
\begin{array}{cccc}
S(e_{11})=e_{22},&S(e_{22})=e_{11},&S(e_{12})=-\xi
e_{12},&S(e_{21})=\xi e_{21},
\end{array}
\end{equation}
\begin{equation}\label{le2}
\begin{array}{cc}
e_{11}^2=e_{22}^2=\alpha,&e_{12}^2=e_{21}^2=0,
\end{array}
\end{equation}
\begin{equation}\label{le3}
\begin{array}{cc}
e_{11}e_{22}=e_{22}e_{11}=\eps,&e_{12}e_{21}=e_{21}e_{12}=0,
\end{array}
\end{equation}
\begin{equation}\label{le4}
\begin{array}{cc}
e_{12}e_{11}=\xi e_{11}e_{12},&e_{21}e_{11}=\xi e_{11}e_{21},
\end{array}
\end{equation}
\begin{equation}\label{le5}
\begin{array}{cc}
e_{12}e_{22}=-\xi e_{22}e_{12},&e_{21}e_{22}=-\xi e_{22}e_{21}.
\end{array}
\end{equation}
In particular we have that
\begin{equation}
\D(e_{11}e_{12})=e_{11}e_{12}\ot\eps+\alpha\ot e_{11}e_{12},
\end{equation}
\begin{equation}
\D(e_{11}e_{21})=e_{11}e_{21}\ot\alpha+\eps\ot e_{11}e_{21},
\end{equation}
i.e., $e_{11}e_{12}$ and $e_{11}e_{21}$ are the non-trivial
skew-primitives of $\cA$.
\end{lema}

\begin{proof}
Since $\cA=(\cA''_4)^*$, the multiplication of $\cA$ is given by
the convolution product and the antipode of $\cA$ by
$S(a)=a\circ\cS $ for all $a\in\cA$, with $\cS$ the antipode of
$\cA''_4$.

Note that $\{g^nx^m\mid0\leq n\leq3,0\leq m\leq1\}$ is a basis of
$\cA''_4$, $\cS(g)=g^{-1}$ and $\cS(x)=-xg^{-1}$ \cite{stefan}.
Then, by \ref{las representations de cA dual}, we have
$$
S(e_{11})(g^n)=e_{11}(\cS(g^n))=e_{11}(g^{-n})=\xi^{-n}=(-\xi)^n=e_{22}(g^n)
$$
$$
\mbox{and
}\,S(e_{11})(g^nx)=e_{11}(\cS(g^nx))=e_{11}(-xg^{-n-1})=0=e_{22}(g^nx),
$$
then $S(e_{11})=e_{22}$. Similarly, we prove $S(e_{22})=e_{11}$.
Clearly, $S(e_{12})(g^n)=S(e_{21})(g^n)=0$ for all $n$. Moreover,
by \ref{las representations de cA dual},
$$
\begin{array}{ll}
S(e_{12})(g^nx)&=e_{12}(\cS(g^nx))=e_{12}(-xg^{-n-1})
\\&=-2(-\xi)^{-n-1}=-2\xi\xi^{n}=-\xi e_{12}(g^nx),
\end{array}
$$
then $S(e_{12})=-\xi e_{12}$. Similarly, we prove $S(e_{21})=\xi
e_{21}$ and (\ref{le1}) follows.

\smallbreak Since $g$ is a group-like, we have
$$
e_{11}^2(g^n)=(e_{11}(g^n))^2=\xi^{2n}=(-1)^n=\alpha(g^n)\mbox{ and}
$$
$$
e_{12}^2(g^n)=(e_{12}(g^n))^2=0,
$$
for all $n$. Since $x$ is a $(1,g)$-primitive, then
$$
e_{11}^2(g^nx)=e_{11}(g^nx)e_{11}(g^{n+1})+e_{11}(g^n)e_{11}(g^nx)=0=\alpha(g^nx)\mbox{
and}
$$
$$
e_{12}^2(g^nx)=e_{12}(g^nx)e_{12}(g^{n+1})+e_{12}(g^n)e_{12}(g^nx)=0.
$$
Therefore $e_{11}^2=\alpha$ and $e_{12}^2=0$. Analogously, we
prove $e_{22}^2=\alpha$ and $e_{21}^2=0$ and (\ref{le2}) follows.

\smallbreak From similar calculations it follows that
$e_{12}e_{21}=0=e_{21}e_{12}$. Then, by (\ref{le1}) and the
definition of antipode, we get $e_{11}e_{22}=\e=e_{22}e_{11}$ and
(\ref{le3}) follows.

\smallbreak If (\ref{le4}) holds, (\ref{le5}) also holds. In fact,
(\ref{le5}) follows from (\ref{le1}) and (\ref{le4}). Since $g\in
G(\cA''_4)$, $e_{11}e_{12}(g^n)=e_{12}e_{11}(g^n)=0$. On the other
hand
$$
\begin{array}{rl}
e_{11}e_{12}(g^nx)&=e_{11}(g^nx)e_{12}(g^{n+1})+e_{11}(g^n)e_{12}(g^nx)\\
\noalign{\smallskip}
&=e_{11}(g^n)e_{12}(g^nx)=\xi^n\xi^n2=(-1)^n2\quad\mbox{and}
\\ \noalign{\smallskip}
e_{12}e_{11}(g^nx)&=e_{12}(g^nx)e_{11}(g^{n+1})+e_{12}(g^n)e_{11}(g^nx)\\
\noalign{\smallskip}
&=e_{12}(g^nx)e_{11}(g^{n+1})=\xi^n2\xi^{n+1}=(-1)^n2\xi,
\end{array}
$$
then $e_{12}e_{11}=\xi e_{11}e_{12}$. Also
$e_{11}e_{21}(g^n)=e_{21}e_{11}(g^n)=0$ and
$$
\begin{array}{ll}
e_{11}e_{21}(g^nx)&=e_{11}(g^nx)e_{21}(g^{n+1})+e_{11}(g^n)e_{21}(g^nx)\\
\noalign{\smallskip} &=e_{11}(g^n)e_{21}(g^nx)=-\xi^n(-\xi)^n=-1,
\\\noalign{\smallskip}
e_{21}e_{11}(g^nx)&=e_{21}(g^nx)e_{11}(g^{n+1})+e_{21}(g^n)e_{11}(g^nx)\\
\noalign{\smallskip}
&=e_{21}(g^nx)e_{11}(g^{n+1})=-(-\xi)^n\xi^{n+1}=-\xi,
\end{array}
$$
then $e_{21}e_{11}=\xi e_{11}e_{21}$ and (\ref{le4}) follows.

\smallbreak Since $\D$ is an algebra morphism, $e_{11}e_{12}$ and
$e_{11}e_{21}$ are skew-primitive by \eqref{le2} and \eqref{le3}.
Since $(\e-\alpha)(g)\neq0$, they also are non-trivial.
\end{proof}

Let $T$ be the Hopf subalgebra of $\cA$ generated by $\alpha$ and
$y:=e_{11}e_{21}$. Note that it is isomorphic to $T_{-1}$. Let
$C_2=\langle c\rangle$ be the cyclic group of order two. We end
this section with the following lemma that will be needed later.

\begin{lema}\label{proyections desde la unica}
(i) If $\pi:\cA\rightarrow T_{-1}$ is a morphism of Hopf algebras,
then $\pi(\cA)\subseteq k[G(T_{-1})]$ and
$T\subseteq\cA^{\co\pi}$.

(ii) $\cA$ fits into an exact sequence of Hopf algebras
$T\overset{\imath}\hookrightarrow\cA\overset{\psi}\twoheadrightarrow
k[C_2]$.
\end{lema}

\begin{proof}
The unique group-like of order two of $\cA''_4$ is central. Then
$\cA''_4$ cannot have a Hopf subalgebra isomorphic to $T_{-1}$;
implying that $\pi$ is not an epimorphism. Hence
$\pi(\cA)\subseteq k[G(T_{-1})]$.

\smallbreak Clearly, $k[G(T_{-1})]$ does not have nilpotent
elements. Then $\pi(e_{12})=\pi(e_{21})=0$, by (\ref{le2}).
Therefore $\pi(e_{11}),\pi(e_{22})\in G(T_{-1})$ and by
(\ref{le2}), $\pi(\alpha)=1$. In particular, it follows that
$T\subseteq\cA^{co\pi}$.

\smallbreak Let $\psi: \cA \to k[C_2]$ be the Hopf algebra
epimorphism induced by the inclusion of the Hopf subalgebra of
$\cA''_4 $ generated by the central group-like element $g^{2}$. By
the paragraph above and \ref{resumen sobre ext nec}, (ii) follows.
\end{proof}

\section{Hopf algebras generated by simple coalgebras}
\label{Hopf algebras generated by simple
coalgebras}

In this section, we discuss some consequences of results of
\cite{natale} and \cite{stefan}. The following theorem will be used
later.

\begin{theorem}\label{stefan 1.4 b}\cite[Thm. 1.4. b)]{stefan}
Let $f$ be a coalgebra automorphism of $C = \coM$ of finite order
$n$. Then there is a comatrix basis $\{e_{ij}\mid1\leq i,j\leq2
\}$ of $C$ and a root of unity $\omega$ such that $f(e_{ij}) =
\omega^{i-j}e_{ij}$ and $\ord \omega = n$.
\end{theorem}

\begin{lema}\label{truco util}
Let $\pi:H\rightarrow K$ be a morphism of finite-dimensional Hopf
algebras such that $\pi(g)=1$ for some $g\in\GH,\,g\neq1$. Suppose
that $H$ is generated by $C$ and $1$ as an algebra, where $C$ is a
simple subcoalgebra of dimension $4$ stable by $L_g$. Then
$\pi(H)\subseteq k[G(K)]$. \smallbreak The same holds true with
$R_g$ instead of $L_g$; or with $\adl(g)$ or $\ad_r(g)$ if
$g\notin{\mathcal Z}(H)$.
\end{lema}

\begin{proof}
First, we claim that $L_{g\mid C}\neq\id_{C}$. Indeed, let
$\{e_{ij}\mid1\leq i,j\leq2\}$ be a comatrix basis of $C$, then
$1= e_{11}S(e_{11}) + e_{12}S(e_{21})$. If $L_{g\mid C}=\id_C$,
multiplying on both sides of the equality we get that $g=1$, a
contradiction. Since $L_{g\mid C}$ is a coalgebra automorphism of
$C$, applying \ref{stefan 1.4 b}, we get $\{e_{ij}\mid1\leq
i,j\leq2\}$ a comatrix basis of $C$ such that
\begin{equation}\label{nec 5}
L_g(e_{ij})=ge_{ij}=\omega^{i-j}e_{ij},\quad\mbox{with}\quad
\omega\in k,\,\ord(\omega)=\ord(L_{g\mid C})>1.
\end{equation}
Applying $\pi$ on both sides of (\ref{nec 5}), we get
$\pi(e_{12})=\pi(e_{21})=0$. Then $\pi(e_{11}),\pi(e_{22})\in
G(K)$ and therefore $\pi(H)\subseteq k[G(K)]$.

\smallbreak The proof for $R_g$, $\adl(g)$ and $\ad_r(g)$ is
similar. Note that $\adl(g)$ and $\ad_r(g)$ cannot be the identity
because $g\notin{\mathcal Z}(H)$.
\end{proof}

\begin{lema}\label{truco util bis}
Let $\pi:H\rightarrow K$ be an epimorphism of finite-dimensional
Hopf algebras and assume that $K$ is non-semisimple. Suppose that
$H$ is generated by $C$ and $1$ as an algebra, where $C$ is a
simple subcoalgebra of $H$ of dimension $4$ stable by $S_H^2$.
Then $\ord S_H^2 = \ord S^2_K$.
\end{lema}

\begin{proof}
By  \ref{stefan 1.4 b}, there is a comatrix basis
$\{e_{ij}\mid1\leq i,j\leq2\}$ of $C$ such that
\begin{equation}\label{nec 5bis}
S_H^2(e_{ij})=\omega^{i-j}e_{ij},\quad\mbox{with}\quad \omega\in
k,\,\ord(\omega)=\ord(S^2_{H\mid C}).
\end{equation}
Applying $\pi$ on both sides of (\ref{nec 5bis}), we get
$S_K^2(\pi(e_{ij}))=\omega^{i-j}\pi(e_{ij})$, that is, at least
one of the numbers $\omega^{\pm 1}$ is an eigenvalue of $S_K^2$.
Indeed, $\pi(e_{12})\neq 0$ or $\pi(e_{21})\neq 0$ since otherwise
$\pi(e_{11}),\pi(e_{22})\in G(K)$ and $K$ would be semisimple,
because $H$ is generated by $C$ and $1$ as an algebra. Then
$\ord(\omega)=\ord(S^2_{H\mid C})=\ord S^2_H$ divides $\ord
S^2_K$.

Finally, since $K^* \hookrightarrow H^{*}$, $(S^2_K)^{\ord S^2_H}
= 1$ and therefore they must be equal.

\end{proof}

The following proposition is due to Natale. It is the key step for
the proof of the last result of this section.

\begin{prop}\label{enunciado de sonia sobre el encaje de H en una
extension}\cite[Prop. 1.3]{natale}. Let $H$ be a
finite-dimensional non-semisimple Hopf algebra. Suppose that $H$
is generated by a simple subcoalgebra of dimension $4$ which is
stable by the antipode. Then $H$ fits into an central exact
sequence $k^G\overset{\imath}\hookrightarrow
H\overset{\pi}\twoheadrightarrow A,$ where $G$ is a finite group
and $A^*$ is a pointed non-semisimple Hopf algebra. $\quad\square$
\end{prop}

\begin{theorem}\label{2n tiene dual punteado}
Let $H$ be a non-semisimple Hopf algebra such that $\dim H$ is odd
or equal to $p^aq^b$, with $p,q$ primes. Suppose that $H$ is
generated by a simple subcoalgebras of dimension $4$ which is
stable by the antipode. If
$$
H_0=\GH\mas\coM\quad\mbox{or}\quad\GH\cap{\mathcal Z}(H)=1
$$
then $H^\ast$ is pointed.
\end{theorem}

\begin{proof}
By \ref{enunciado de sonia sobre el encaje de H en una extension},
$H$ fits into an central exact sequence
$k^G\overset{\imath}\hookrightarrow
H\overset{\pi}\twoheadrightarrow A,$ where $G$ is a finite group
and $A^*$ is a pointed non-semisimple Hopf algebra.

Suppose that $G\neq1$. Since $|G|$ divides $\dim H$ by
\cite{nicozoeler}, $G$ is solvable by the Feit-Thompson Theorem in
the case that $\dim H$ is odd, and the Burnside Theorem in the
other case. Thus $k^G$ has at least one non-trivial group-like.
Let $\alpha\in G(k^G)\subseteq\GH$ be non-trivial.

Suppose that $H_0=\GH\mas\coM$. Since $L_{\alpha}$ is a coalgebra
automorphism of $H$, $L_\alpha$ fixes $\coM$. As $\pi(\alpha)=1$,
by \ref{truco util} it follows that $A$ is generated by
group-likes. In particular, $A$ is semisimple which is a
contradiction. Therefore $G=1$, that is, $H=A$ and $H^*$ is
pointed.

If $\GH\cap{\mathcal Z}(H)=1$ then clearly $G=1$ and $H^*$ is
pointed.
\end{proof}

\smallbreak

\section{Proof of the Main Theorem}\label{Proof of the Main Theorem}
Our first step to prove Theorem \ref{main result} is to describe the
possible coradical of a Hopf algebra of dimension 16 which does not
have the Chevalley property. It turns out that there are $6$
possible coradicals. This leads us to do the proof case by case
according to the type of the coradical.

\begin{definition}
We say that a Hopf algebra $H$ is {\it of type}
$(n_1,n_2,\dots,n_t)$ if the coradical of $H$ is $H_0\simeq
k^{n_1}\mas\coM^{n_2}\cdots\mas{\mathcal M}^*(t,k)^{n_t}$.
\end{definition}

\begin{obs}\label{los duales de biti} Let $H$ be a pointed Hopf
algebra of dimension $16$. Then by \cite[Section 4.2]{biti}, $H^*$
is pointed or it is of type (2,1), (2,2), (2,3) or (4,2).
\end{obs}

\begin{obs}\label{chevalley sii su dual}
Let $H$ be a non-semisimple non-pointed Hopf algebra of dimension
$16$ which has the Chevalley property. Then by \cite[Thm.
5.1]{de1tipo6chevalley}, $H$ is self-dual and of type $(4,1)$.
\end{obs}

\smallbreak

\begin{prop}\label{posibles coradicales}
Let $H$ be Hopf algebra of dimension $16$ which does not have the
Chevalley property. Then $H$ is of type: (1,2), (2,1), (2,2),
(2,3), (4,1) or (4,2).
\end{prop}
\begin{proof}
If $G(H)=1$, then by \cite[Prop. 7.1]{bitidasca} we know that $H$
must be of type (1,2).

Now suppose that $H$ is of type $(|\GH|, n_2, n_3)$ with
$|\GH|>1$. By \cite{nicozoeler}, $|\GH|$ divides $16$. Moreover by
\cite[Lemma 2.1]{andrunatale}, it divides $$\dim\Ho=|\GH|+4\cdot
n_2+9\cdot n_3.$$

\noindent Thus $n_3=0$. If $|\GH|=2$, then $n_2= 1, 2$ or $3$ and
if $|\GH|=4$, then $n_2= 1$ or $2$.
\end{proof}

\smallbreak

We next give some properties of Hopf algebra of dimension $16$
which does not have the Chevalley property. We recall first the
following statement due to Beattie and Dascalescu.

\begin{prop}\label{cota de biti y dasca}\cite[Cor. 4.3]{bitidasca}.
Let $H$ be a finite-dimensional non-cosemisimple Hopf algebra with
$H_0\simeq k[G]\mas{\mathcal M}^*(n_1,k)\mas\cdots\mas{\mathcal
M}^*(n_t,k)$ with $t$ a positive integer, $2\leq n_1\leq\cdots\leq
n_t$, and such that $H$ has no non-trivial skew-primitives. Then
\begin{equation}\label{nec 4}
\dim H>\dim H_1=\dim H_0+\dim P_1\geq
(1+2n_1)|G|+\sum_{i=1}^tn_i^2.\qquad\square
\end{equation}
\end{prop}

\begin{lema}\label{existen casi primitivos}
Let $H$ be a Hopf algebra of dimension 16.
\begin{enumerate}[(i)]
\item If $H$ is of type $(4,1)$ or $(4,2)$ then $H$ has a pointed
Hopf subalgebra $K$ of dimension $8$ such that $G(H) = G(K)$.
\item If $H$ is of type $(2,2)$ or $(2,3)$ then $H$ has a Hopf
subalgebra isomorphic to $T_{-1}$. \item If $H$ is of type $(2,1)$
and $H^*$ is non-pointed then $H$ has a Hopf subalgebra isomorphic
to $\cA$ (see Subsection \ref{la unica}). In particular, it
contains a Hopf subalgebra isomorphic to $T_{-1}$. \item If $H$ is
of type $(2,n)$ with $1\leq n\leq 3$, then $G(H)\cap
\mathcal{Z}(H) = 1$. If it is of type $(4,n)$ with $1\leq n\leq
2$, then $|G(H)\cap \mathcal{Z}(H)| \leq 2$.
\end{enumerate}
\end{lema}
\begin{proof}
If $H$ is of type $(2,2),\,(2,3),\,(4,1)$ or $(4,2)$ then $H$ has
a non-trivial skew-primitive. Otherwise, we can apply \ref{cota de
biti y dasca} to $H$ and we obtain a contradiction by (\ref{nec
4}). Let $K$ be the Hopf subalgebra of $H$ generated by $\GH$ and
$\PH$. By \cite[Lemma 5.5.1]{mongomeri}, $K$ is pointed and $\dim
K>|\GH|$.

If $|\GH|=4$, then $\dim K= 8$ by \cite{nicozoeler}. This proves
$(i)$.

If $|\GH|=2$, by \cite{nicozoeler} and \cite{stefan} $K$ is
isomorphic to $T_{-1}$ or $\cA_2$ (see Section \ref{la
clasificacion de stefan}). But $\cA_2$ has a Hopf subalgebra
isomorphic to $T_{-1}$ by \ref{para cuando necesite una sweedler}.
This proves $(ii)$.

Let $H$ be as in $(iii)$ and let $C$ be the unique simple
subcoalgebra of $H$ of dimension $4$. The Hopf subalgebra $K$
generated by $C$ has dimension $8$ or $16$. We claim that $\dim
K=8$ and therefore $K\simeq \cA$ by \cite{stefan}. In fact, if
$K=H$, then $H^*$ is pointed by \ref{2n tiene dual punteado}. But
this cannot happen, since $H^{*}$ is non-pointed by hypothesis;
then $\dim K=8$.

Finally, we prove $(iv)$. If $H$ is of type $(2,n)$ with $1\leq
n\leq 3$, then by $(ii)$ and $(iii)$ it contains a Sweedler
subalgebra $T_{-1}$. In particular, $G(T_{-1}) = G(H)$ and the
claim follows since $G(T_{-1})\cap \mathcal{Z}(T_{-1}) = 1$. If
$H$ is of type $(4,n)$ with $1\leq n\leq 2$, then the assertion
follows by $(i)$ and Section \ref{la clasificacion de stefan}.
\end{proof}

\smallbreak

\emph{We assume for the rest of the paper that $H$ is a Hopf
algebra of dimension $16$ which has not the Chevalley property. }

\smallbreak Note that by \ref{chevalley sii su dual}, if $H^{*}$
is non-pointed, then it does not have the Chevalley property
either.

\smallbreak In the next subsections we prove that $H$ cannot be of
type $(1,2)$ -- see \ref{16 tiene gr likes}; also, if $H$ is of
type $(s,t)$ then $H^*$ has the Chevalley property, for each
$(s,t)$ with $s>1$, according to \ref{posibles coradicales} -- see
\ref{no existe G4 mas coM}; \ref{G4 mas coM mas coM es dual
punteada}; \ref{G2 mas coM or mas coM mas coM es dual punteada}
and \ref{G2 mas coM mas coM}. Then the Theorem \ref{main result}
is proved.

\subsection{Type $(1,2)$}

\begin{obs}\label{C y D generan H entonces C tambien}
Let $H$ be a finite-dimensional Hopf algebra generated by two
simple subcoalgebras $C$ and $D$ such that $S(C) =D$. Then $C$ and
$1$ generate $H$ as an algebra.

Indeed, the subalgebra $A$ of $H$ generated by $C$ and $1$, is a
sub-bialgebra. Since $\dim H<\infty$, $A$ is a Hopf subalgebra and
then $D=S(C)\subseteq A$.
\end{obs}

\begin{prop}\label{16 tiene gr likes}
$H$ cannot be of type (1,2).
\end{prop}

\begin{proof}
Suppose that $H$ is  of type $(1,2)$. Then $H^*$ is not
cosemisimple and hence it is non-semisimple by \cite{larad}.
Moreover, it must be non-pointed by \ref{los duales de biti} and
$H^*$ does not have the Chevalley property by \ref{chevalley sii
su dual}. Thus we can apply \ref{posibles coradicales} to $H^*$.

Let $C$ and $D$ be the simple subcoalgebras of $H$ of dimension 4.
If $C$ (and hence $D$) is stable by $S$, then the Hopf subalgebra
$K$ generated by $C$ is $H$. Otherwise, $K$ should be isomorphic
to $\cA$ or semisimple by the classification of $8$-dimensional
Hopf algebras. In both cases we would have $1\neq G(K) \subseteq
G(H)$. Hence by \ref{2n tiene dual punteado}, $H^*$ is pointed,
which is a contradiction. Therefore $S$ permutes $C$ with $D$, and
so $H$ is generated by $C$ and $D$ as an algebra by
\cite{nicozoeler}. In particular, $C$ and $1$ generate $H$ as an
algebra by \ref{C y D generan H entonces C tambien}.

We claim now that $S^4=\id$. Indeed, if $G(H^*)=1$ the claim
follows from Radford's formula for $S^4$. If $G(H^*)\neq 1$ then
by \ref{existen casi primitivos}, $H^*$ has a Hopf subalgebra $K$
such that $K^*$ is non-semisimple and $S_K^4=\id_K$. Then there
exists an epimorphism of Hopf algebras $\pi:H\rightarrow K^*$ and
by \ref{truco util bis}, the claim follows.

Therefore, by \cite[Prop. 5.3]{bitidasca} $H$ has a simple
subcoalgebras of dimension 4 stable by $S$, which is a
contradiction to the fact that $S$ permutes the simple
subcoalgebras.
\end{proof}

\subsection{Type $(4,1)$}

\begin{prop}\label{no existe G4 mas coM}
$H$ cannot be of type (4,1).
\end{prop}

\begin{proof}Let
$K$ be the Hopf subalgebra of $H$ generated by $C$, the simple
subcoalgebra of $H$ of dimension $4$. Note that $\dim K \neq 16$
since otherwise $H^*$ would be pointed by \ref{2n tiene dual
punteado}, and this would contradict \ref{los duales de biti}.
Hence $\dim K=8$ and $K$ is non-pointed. If $K$ is semisimple,
then by \cite{larad} $K$ is cosemisimple and $K = H_{0}$ by
counting, a contradiction to the hypothesis on $H$. Hence
$K\simeq\cA$.

As $G(\cA)=C_2$, there exists $g\in\GH - G(K)$. Since $C$ is the
unique simple subcoalgebra of $H$ of dimension $4$, $C$ is stable
by $L_g$, which is a coalgebra automorphism of $H$. Let $B$ be the
Hopf subalgebra generated by $K$ and $g$. Since the multiplication
is associative and $C$ generates $K$ as an algebra, we have that
$B=K+k[g]$ as a vector spaces; in particular $8<\dim B<16$. This
is impossible by \cite{nicozoeler}.
\end{proof}

We finish this subsection with a criterion that helps us to know
when $H^*$ is pointed. The key of argument comes from the proof of
\cite[Thm. 2.1]{gaston}.

\begin{lema}\label{como el teo de gaston}
Suppose that $H$ fits into an exact sequence
$K\overset{\imath}\hookrightarrow H\overset{\pi}\twoheadrightarrow
k[C_2]$, where $K^\ast$ is pointed. Then $H^\ast$ is pointed.
\end{lema}

\begin{proof}
By subsection \ref{ext de Hopf}, $H$ is isomorphic as algebra to a
crossed product $K\#_{\rightharpoonup,\sigma} k[C_2]$. Denote by
$g$ the generator of $C_2$. Then the weak action $l(g):K\mapsto
K,\,a\mapsto (g\rightharpoonup a)$ is an isomorphism of algebras.
In particular $\Rad K$ is stable by $l(g)$ and therefore $\Rad
K\#_{\rightharpoonup,\sigma} k[C_2]$ is a nilpotent $C_2$-graded
ideal. This implies that $\Rad K\#_{\rightharpoonup,\sigma}
k[C_2]\subseteq\Rad H$. Besides $H/(\Rad
K\#_{\rightharpoonup,\sigma}k[C_2])\simeq (K/\Rad
K)\#_{\overline{\rightharpoonup},\overline{\sigma}} k[C_2]$ is a
semisimple algebra, by \cite[Thm. 7.4.2]{mongomeri}. Then $\Rad
H\subseteq\Rad K\#_{\rightharpoonup,\sigma}k[C_2]$, and hence
$H/\Rad H\simeq (K/\Rad
K)\#_{\overline{\rightharpoonup},\overline{\sigma}} k[C_2]$.

We conclude the proof examining the dimension of $(H^*)_0$. Since
$K^{*}$ is pointed, $\dim (K^*)_0=\dim (K/\Rad K)=2,4$ or $8$ and
therefore $\dim (H^*)_0=\dim(H/\Rad H)=4,8$ or $16$. If $\dim
(H^*)_0=4$, $H^*$ is clearly pointed. If $\dim (H^*)_0=8$, $H^*$
is pointed by \ref{posibles coradicales}, \ref{chevalley sii su
dual} and \ref{no existe G4 mas coM}. Since $H$ is non-semisimple,
$\dim (H^*)_0=16$ cannot occur.
\end{proof}

\subsection{Type $(4,2)$}

Throughout this subsection $C$ and $D$ are the two simples
subcoalgebras of $H$ of dimension $4$. We show in a series of
lemmata that the dual of a Hopf algebra of type $(4,2)$ is
pointed. First we compute the order of $S$. Note that $S^{2}$
preserves $C$ and $D$.

\begin{lema}\label{S en G4 mas coM mas coM}
If $H$ is of type (4,2) then $\ord S^2_{\mid C}=\ord S^2_{\mid
D}=2$. Moreover, $\ord S=4$.
\end{lema}

\begin{proof}
Let $K$ be the pointed Hopf subalgebra of $H$ of dimension $8$
given in \ref{existen casi primitivos}. $(i)$. Then $H=K\mas C\mas
D$ is a direct sum of $S^{2}$-stable subspaces. Since $H$ and $K$
are non-semisimple, $\Tr(S^2)=\Tr(S^2_{\mid K})=0$. Moreover, by
\cite[Lemma 3.2]{larad}, we have that $\Tr(S^2_{\mid\coM})\geq 0$,
hence $\Tr(S^2_{\mid D})=\Tr(S^2_{\mid C})=0$.

Let $\{e_{ij}\mid1\leq i,j\leq2\}$ be a comatrix basis of $C$ such
that: $S^2(e_{ij})=\omega^{i-j}e_{ij}$ with $\omega\in k$ and
$\ord\omega=\ord S^2_{\mid C}=n$ by \ref{stefan 1.4 b}. Since
$0=\Tr(S^2_{\mid C})=2+\omega+\omega^{-1}$, multiplying by
$\omega$ on both sides we get $0=1+2\omega+\omega^2=(1+\omega)^2$.
Hence $\omega=-1$ and hence $\ord S^2_{\mid C}=2$. The same holds
true for $D$ instead of $C$.

Finally, $\ord S=4$ since by \cite{stefan}, $\ord S_{\mid K}=4$.
\end{proof}

\smallbreak

\begin{lema}\label{gr like central en G4 mas coM mas coM}
Let $H$ be of type (4,2) and suppose that there exists
$g\in\GH\cap\mathcal{Z}(H)$ and $H$ is generated as an algebra by
$C$ and $1$. Then $H^*$ is pointed.
\end{lema}

\begin{proof}
By \ref{existen casi primitivos} $(iv)$, the order of $g$ is $2$.
Then $H$ fits into an exact sequence
$k[C_2]\overset{\imath}\hookrightarrow
H\overset{\pi}\twoheadrightarrow K$, where $K= H/k[C_2]^{+}H$.
Since $H$ is non-semisimple, $K$ must be non-semisimple by
\cite[Thm. 7.4.2]{mongomeri}. If $K$ is pointed, then $H^{*}$ is
pointed by \ref{como el teo de gaston}. If $K$ is not pointed,
then $K\simeq \cA$ by the classification given by \cite{stefan},
see Section \ref{la clasificacion de stefan}.

Suppose that $\GH =\langle c \rangle$ is a cyclic group of order
$4$. Then $L_g =L_{c^{2}} = L_{c}^{2}$ must fix $C$ and $\pi(g) =
1$ because $G(\cA)\cap \mathcal{Z}(\cA) = 1$. Then by \ref{truco
util}, we obtain a contradiction. Therefore $\GH\simeq C_2\times
C_2$. Hence, by \ref{existen casi primitivos} $(i)$ and
\cite{stefan} $H$ has a Hopf subalgebra isomorphic to $\cA_{2,2}$.

Since $\cA_{2,2}$ and $\cA$ are not isomorphic, it follows that
$\pi(\cA_{2,2})\subseteq T_{-1}\subseteq\cA$.

Let $T_{-1}\overset{\imath}\hookrightarrow\cA
\overset{\psi}\twoheadrightarrow k[C_2]$ be the exact sequence
given in \ref{proyections desde la unica}. Then
$\psi\circ\pi:H\rightarrow k[C_2]$ is an epimorphism of Hopf
algebras and $\cA_{2,2}\subseteq\,H^{\co(\psi\circ\pi)}$. Then by
\ref{resumen sobre ext nec}, $H$ fits into the exact sequence
$\cA_{2,2}\overset{\imath}\hookrightarrow H
\overset{\psi\circ\pi}\twoheadrightarrow k[C_2]$. Since
$\cA_{2,2}$ is self-dual, $H^\ast$ is pointed by \ref{como el teo
de gaston}.
\end{proof}

\smallbreak

\begin{prop}\label{G4 mas coM mas coM es dual punteada}
Let $H$ be of type (4,2). Then $H^*$ is pointed.
\end{prop}

\begin{proof}
We divide the proof in two cases, according to the action of $S$
on $\{C,D\}$.

\smallbreak

{\it \underline {Case 1}: $C$ and $D$ are stable by $S$.}

Let $K$ be the Hopf subalgebra of $H$ generated by $C$. First,
suppose that $\dim K=8$. Then $K\simeq\cA$ because $K$ is
non-semisimple by \ref{S en G4 mas coM mas coM} and \cite{larad}.
Let $g\in\GH- G(K)$. We claim that $K$ is a normal Hopf subalgebra
and hence $H^\ast$ is pointed by \ref{como el teo de gaston}.
Indeed, if $L_g$ and $R_g$ fix $C$ we get a contradiction as in
the proof of \ref{no existe G4 mas coM}. Thus we may assume that
$L_g(C)=D$ and hence $L_g(D)=C$. Applying $S$ to the second
equality, it follows that $R_{g^{-1}}(D)=C$. Then $C$ and
therefore $K$ are stable by $\adl(g)$. Since $\ord g<\infty$, $K$
also is stable by $\ad_r(g)$. Since by \cite{nicozoeler}, $K$ and
$g$ generate $H$ as an algebra, it follows that $K$ is normal.

Now, suppose that $K=H$. Then by \ref{enunciado de sonia sobre el
encaje de H en una extension}, $H$ fits into the central exact
sequence $ k^G\overset{\imath}\hookrightarrow
H\overset{\pi}\twoheadrightarrow A$, where $G$ is a finite group
and $A^\ast$ is pointed non-semisimple. Since $A$ is
non-semisimple, $|G|\neq8,16$. Moreover, by \ref{existen casi
primitivos} $(iv)$ and \cite{stefan}, $|G|\neq4$. If $|G|= 2$,
then $H^*$ is pointed by \ref{gr like central en G4 mas coM mas
coM}. If $|G|=1$, then $H=A$ and hence $H^*$ is pointed.

\smallbreak

{\it \underline {Case 2}: $C$ and $D$ are permuted by $S$.}

Note that by \cite{nicozoeler}, $C$ and $D$ generate $H$ as an
algebra. Hence by \ref{C y D generan H entonces C tambien}, $C$
and $1$ generate $H$ as an algebra.

Suppose that $H^*$ is non-pointed. Then, by \ref{existen casi
primitivos}, there exists $\pi:H\twoheadrightarrow B$ an
epimorphism Hopf algebras, where we can assume that $B$ is
isomorphic to: $T_{-1}$, $\cA$ or $\cA'''_{4,i}$. Indeed, $H^{*}$
cannot be of type $(4,1)$ by \ref{no existe G4 mas coM}. If it is
of type $(4,2)$, then it contains a pointed Hopf subalgebra $L$ of
dimension $8$ isomorphic to $\cA'_{4} = (\cA'''_{4,i})^{*}$ or
$\cA''_{4} = \cA^{*}$, or $L$ contains a Sweedler algebra, see
\ref{para cuando necesite una sweedler}. Finally, if $H^{*}$ is of
type $(2,n)$, then by \ref{existen casi primitivos} it contains a
Sweedler algebra and the claim follows.

\smallbreak We first assume that $G(H)$ is cyclic. Say
$\GH=\langle g\rangle$. Then $L_{g^2}(C)=C$. If
$\ord(\pi(g))\leq2$, then $\pi(g^2)=1$ and by \ref{truco util}
$\pi(H)\subseteq k[G(B)]$, which is impossible because $\pi$ is an
epimorphism and $B$ is non-semisimple. Hence $\ord g = 4$. Since
$|G(T_{-1})|=|G(\cA)|=2$, $B$ must be isomorphic to $\cA'''_{4,i}$
and $\pi(g)$ generates $G(B)$. We now proceed as in the proof of
\cite[Lemma 2.7]{natale}: let $B^+$ and $B^-$ denote the following
subspaces of $B$
\begin{eqnarray*}
B^+:=\{b\in B\mid S^2(b)=b\}&\mbox{and}&B^-:=\{b\in B\mid
S^2(b)=-b\}.
\end{eqnarray*}
From the definition of $\cA'''_{4,i}$ (see subsection \ref{la
clasificacion de stefan}), it follows that $B^+=k[G(B)]$.

Let $\{e_{ij}\mid1\leq i,j\leq2\}$ be a comatrix basis of $C$ such
that $S^2(e_{ij})=(-1)^{i-j}e_{ij}$ by \ref{stefan 1.4 b}. Then
$\pi(e_{11}),\pi(e_{22})\in B^+=k[G(B)]$ and hence
$$\D(\pi(e_{11}))=\pi(e_{11})\ot\pi(e_{11})+\pi(e_{12})\ot\pi(e_{21})\in
B^+\ot B^+.
$$
Since $\pi(e_{11})\ot\pi(e_{11})\in B^+\ot B^+$ then
$\pi(e_{12})\ot\pi(e_{21})\in B^+\ot B^+$. But $\pi(e_{12})$,
$\pi(e_{21})\in B^-$, which forces $\pi(e_{12})=0$ or
$\pi(e_{21})=0$. Then $\pi(e_{11}),\pi(e_{22})\in G(B)$. We may
assume that $\pi(e_{21})=0$. Let $n\in\NN$ such that
$\pi(g^ne_{11})=1$. From
\begin{eqnarray*}
\D(g^ne_{11})=g^ne_{11}\ot g^ne_{11}+g^ne_{12}\ot g^ne_{21},
\\ \smallskip
\D(g^ne_{21})=g^ne_{21}\ot g^ne_{11}+g^ne_{22}\ot g^ne_{21},
\end{eqnarray*}
it follows that $1,g^ne_{11},g^ne_{21}\in H^{co\pi}$ and then
$2=\dim H^{co\pi}\geq 3$ (the equality follows from \cite[Thm.
2.4]{sch}), which is impossible. Then $H^*$ is pointed if
$\GH\simeq C_4$.

\smallbreak If $\GH\simeq C_2\times C_2$, there exists $g\in\GH$,
$g\neq 1$ such that $\pi(g)=1$, because $G(B)$ is cyclic. We claim
that $\adl(g)(C)=C$. Indeed, note that either
\begin{equation}\label{nec 6}
L_g(C)=C \Leftrightarrow L_g(D)=D\,\mbox{(and applying
$S$)}\Leftrightarrow R_g(C)=C\quad\mbox{or}
\end{equation}
\begin{equation}\label{nec 7}
L_g(C)=D\,\mbox{(and applying $S$)}\Leftrightarrow R_g(D)=C.
\end{equation}
In any case we obtain that $\adl(g)(C)=C$. If $g\in{\mathcal
Z}(H)$, then $H^{*}$ is pointed by \ref{gr like central en G4 mas
coM mas coM}. If $g\notin{\mathcal Z}(H)$, then $H^{*}$ is pointed
by \ref{truco util}. In both cases we get a contradiction to the
assumption that $H^{*}$ is non-pointed.

\end{proof}

\smallbreak

\subsection{Type $(2,n)$}
In this last subsection we prove that if $H$ is of type $(2,n)$,
$1\leq n\leq 3$, then $H^{*}$ is pointed.

\begin{prop}\label{G2 mas coM or mas coM mas coM es dual punteada}

(i) Let $H$ be of type (2,1) or (2,3). Then $H^*$ is pointed.

(ii) If $H$ is of type (2,2) and has a simple subcoalgebra $C$ of
dimension $4$ stable by the antipode, then $H^*$ is pointed.
\end{prop}

\begin{proof} If $H$ is a Hopf algebra satisfying the
hypothesis of $(i)$, then it contains a simple subcoalgebra of
dimension $4$ stable by the antipode. This is clear when $H$ is of
type $(2,1)$ and when $H$ is of type $(2,3)$, the claim follows
since $\ord S$ is a power of $2$ by Radford's formula. We denote
by $C$ such a simple subcoalgebra.

We prove $(i)$ and $(ii)$ simultaneously. Let $K$ be the Hopf
subalgebra generated by $C$. By \cite{nicozoeler}, $K=H$ or
$K\simeq\cA$, since $K$ is non-pointed by construction and
non-semisimple because $|G(K)| \leq |G(H)|= 2$.

If $K=H$, then $H^*$ is pointed by \ref{2n tiene dual punteado},
and the claim is proved.

Now let $K\simeq\cA$ and assume that $H^*$ is non-pointed. Since
$H^*$ does not have the Chevalley property, $H^*$ must be of type
$(2,n)$ by \ref{posibles coradicales}, \ref{16 tiene gr likes},
\ref{no existe G4 mas coM} and \ref{G4 mas coM mas coM es dual
punteada}. Then by \ref{existen casi primitivos}, there exists
$\pi:H\rightarrow T_{-1}$ an epimorphism of Hopf algebras. Now we
restrict $\pi$ to $K\simeq \cA$. Then by  \ref{proyections desde
la unica} $(i)$, $K^{\co \pi}$ contains a copy of $T_{-1}$. Hence
by \ref{resumen sobre ext nec}, $H^{\co \pi} \simeq T_{-1}$ and
$H$ fits into an exact sequence
$T_{-1}\overset{\imath}\hookrightarrow
H\overset{\pi}\twoheadrightarrow T_{-1}$. But this cannot occur,
since by \ref{lo que creia que era solo el tensor} $H$ should be
pointed. Then $H^*$ must be pointed.
\end{proof}

\begin{prop}\label{G2 mas coM mas coM}
Let $H$ be of type (2,2). Then $H^*$ is pointed.
\end{prop}

\begin{proof}
Suppose that $H^*$ is non-pointed. Then $H^*$ also is of type
$(2,2)$, by \ref{posibles coradicales}, \ref{16 tiene gr likes},
\ref{no existe G4 mas coM}, \ref{G4 mas coM mas coM es dual
punteada} and \ref{G2 mas coM or mas coM mas coM es dual punteada}.

Let $C$ and $D$ be the two simple subcoalgebras of $H$ of
dimension $4$. By \ref{G2 mas coM or mas coM mas coM es dual
punteada} $(ii)$,  $S$ permutes them and by \cite{nicozoeler}, $H$
is generated as an algebra by $C$ and $D$; in particular $H$ is
also generated as an algebra by $C$ and $1$ by \ref{C y D generan
H entonces C tambien}.

We split the proof of the proposition into several claims.
\begin{af}\label{proyeccion a la sweedler}
(i) There exists $\pi:H\twoheadrightarrow T_{-1}$ an epimorphism of
Hopf algebras.

(ii) Let $1\neq g \in G(H)$. Then $\pi(g)\neq 1$.

(iii) $S^2=\adl(g)$.
\end{af}

Indeed, $(i)$ follows from \ref{existen casi primitivos} $(ii)$
applied to $H^*$. Using \eqref{nec 6} and \eqref{nec 7}, it
follows that $\adl(g)$ fixes $C$ and $D$. Since $\pi$ is an
epimorphism and by \ref{existen casi primitivos} $(iv)$
$g\notin{\mathcal Z}(H)$, we have by \ref{truco util} that
$\pi(g)\neq1$, and $(ii)$ follows.

We next prove $(iii)$. By \ref{stefan 1.4 b} there is a comatrix
basis $\{e_{ij}\mid1\leq i,j\leq2\}$ of $C$  and $\omega\in k$
such that $gS^2(e_{ij})g=\omega^{i-j}e_{ij}$. Applying $\pi$ when
$i=2$ and $j=1$, we get
$$
\omega\pi(e_{21})=\pi(gS^2(e_{21})g)=\pi(g)S^2(\pi(e_{21}))\pi(g)=\pi(e_{21}).
$$
The last equality follows from $(ii)$ and the definition of the
antipode of $T_{-1}$. The same holds true with $e_{12}$ instead of
$e_{21}$. Now, if $\omega\neq1$ then $\pi(e_{12})=\pi(e_{21})=0$,
and hence $\pi(H)\subseteq k[G(T_{-1})]$. A contradiction, since
$\pi$ is an epimorphism. Thus $\omega=1$ and  $(iii)$ follows. The
claim is proved. \smallbreak

Let $\cE:=\{e_{ij}\mid1\leq i,j\leq2\}$ denote a comatrix basis of
$C$ such that $S^2(e_{ij})=ge_{ij}g=(-1)^{i-j}e_{ij}$ given by
\ref{stefan 1.4 b}. Then, as in the proof of \cite[Lemm.
2.7]{natale}, the elements of $\cE$ satisfy
\begin{equation}\label{como en 12 bis}
\pi(e_{12})=0\neq\pi(e_{21})\in\cP(T_{-1})\quad\mbox{or}\quad
\pi(e_{21})=0\neq\pi(e_{12})\in\cP(T_{-1}).
\end{equation}
and
\begin{equation}\label{como en 12}
\pi(e_{11})=\pi(g)\mbox{ and
}\pi(e_{22})=1\quad\mbox{or}\quad\pi(e_{11})=1\mbox{ and
}\pi(e_{22})=\pi(g),
\end{equation}

\smallbreak The following claim is inspired by the proof of
\cite[Prop. 5.3]{bitidasca}.

\begin{af}\label{af como en 14}
If $f_{ij}:=S(e_{ji})$ for $1\leq i,j\leq2$ then
\begin{equation}\label{como en 14}
e_{11}f_{22}=f_{22}e_{11}=e_{22}f_{11}=f_{11}e_{22}=g\quad\mbox{and}
\end{equation}
\begin{equation}\label{como en 14 bis}
e_{12}f_{21}=f_{21}e_{12}=e_{21}f_{12}=f_{12}e_{21}=0.
\end{equation}
\end{af}

Indeed, as in \cite[Prop. 5.3]{bitidasca}, we define
$$
E_{11}:=e_{11}f_{22},\,E_{12}:=e_{12}f_{21},\,E_{21}:=e_{21}f_{12}\mbox{
and }E_{22}=e_{22}f_{11}$$
$$
F_{11}:=f_{11}e_{22},\,F_{12}:=f_{12}e_{21},\,F_{21}:=f_{21}e_{12}\mbox{
and }F_{22}=f_{22}e_{11}.
$$

Note that, as in \cite[Prop. 5.3]{bitidasca}, the coalgebra $E$
generated by the $E_{ij}$'s is stable by $S$. The same holds true
for $F$, the coalgebra generated by the $F_{ij}$'s. Since $S$
permutes $C$ and $D$, $\dim E$ and $\dim F$ are less than $4$. We
claim that neither $1\in E$ nor $1\in F$. Indeed, if
$1=\sum_{ij}a_{ij}E_{ij}$ with $a_{ij}\in k$ then we get a
contradiction by writing:
$$
1=\pi(1)=\sum_{ij}a_{ij}\pi(E_{ij})=(a_{11}+a_{22})\pi(g),
$$
where the last equality follows from \eqref{como en 12} and
\eqref{como en 12 bis}. The same holds true if we suppose $1\in
F$.

Then by \cite[Thm. 2.1]{bitidasca}, $E=F=k\cdot g$ and hence
\eqref{como en 14} and \eqref{como en 14 bis} hold. The claim is
proved.

\begin{af}\label{H contiene cA2}
There exists a Hopf subalgebra of $H$ isomorphic to $\cA_2$ (see
Subsection \ref{la clasificacion de stefan}).
\end{af}

Indeed, let $x:=f_{11}e_{12}$. Since
$0=\e(e_{12})=f_{11}e_{12}+f_{21}e_{22}$, it follows that
$x=-f_{21}e_{22}$. Then
$$
\begin{array}{l}
\D(x)=\D(f_{11})\D(e_{12})\\
\noalign{\smallskip}=f_{11}e_{11}\ot f_{11}e_{12}+f_{12}e_{11}\ot
f_{21}e_{12}+f_{11}e_{12}\ot f_{11}e_{22}+f_{12}e_{12}\ot
f_{21}e_{22}\\
\noalign{\smallskip}=f_{11}e_{11}\ot x+f_{12}e_{12}\ot 0+x\ot
g+f_{12}e_{12}\ot (-x)\quad[\mbox{by }\,\eqref{como en 14 bis}\mbox{ and }\eqref{como en 14}]\\
\noalign{\smallskip}=(f_{11}e_{11}-f_{12}e_{12})\ot x+x\ot g\\
\noalign{\smallskip}=1\ot x+x\ot g \qquad\qquad\qquad\qquad[\mbox{by
}\,1=\e(f_{11})=m(\id\ot S)\D(f_{11})].
\end{array}
$$
Moreover, since $f_{11}$ is invertible and $e_{12}\neq0$, it follows
that $x\neq0$.

Also, let $y:=f_{22}e_{21}$. Since
$0=\e(e_{21})=f_{12}e_{11}+f_{22}e_{21}$, it follows that
$y=-f_{12}e_{11}$. Then
$$
\begin{array}{l}
\D(y)=\D(f_{22})\D(e_{21})\\
\noalign{\smallskip}=f_{21}e_{21}\ot f_{12}e_{11}+f_{22}e_{21}\ot
f_{22}e_{11}+f_{21}e_{22}\ot f_{12}e_{21}+f_{22}e_{22}\ot
f_{22}e_{21}\\
\noalign{\smallskip}=f_{21}e_{21}\ot (-y)+y\ot g+f_{21}e_{22}\ot
0+f_{22}e_{22}\ot y\quad[\mbox{by }\,\eqref{como en 14}\mbox{ and }\eqref{como en 14 bis}]\\
\noalign{\smallskip}=(f_{22}e_{22}-f_{21}e_{21})\ot y+y\ot g\\
\noalign{\smallskip}=1\ot y+y\ot g\qquad\qquad\qquad\qquad[\mbox{by
}\,1=\e(f_{22})=m(\id\ot S)\D(f_{22})].
\end{array}
$$
Since $f_{22}$ is invertible and $e_{21}\neq0$, it follows that
$y\neq0$.

\smallbreak If $\{1-g,x,y\}$ are linearly independent then the claim
follows. In fact, the Hopf subalgebra generated by $\{g,x,y\}$ must
be of dimension $8$ by \cite{nicozoeler}, and by Subsection \ref{la
clasificacion de stefan} it must be isomorphic to $\cA_2$.

We next prove that $\{1-g,x,y\}$ are linearly independent. Let
$a,b,c\in k$ such that $0=a(1-g)+bx+cy$. Appliying $\pi$ we get that
$-a(1-\pi(g))=b\pi(x)+c\pi(y)$. But by the Claim \ref{proyeccion a
la sweedler} $(ii)$, we have that $\pi(g)\neq 1$ and by \eqref{como
en 12 bis} and \eqref{como en 12}, $\pi(x) = 0$ or $\pi(y) = 0$.
Hence, $\pi(g)$ is the group-like element of $T_{-1}$ and $\pi(x) $
or $\pi(y)$ is the skew-primitive. Then $a=0$ and $b=0$ or $c=0$.
But if $b\neq0$ or $c\neq0$, then $x=0$ or $y=0$; a contradiction.
Hence $a=b=c=0$ that is, $\{1-g,x,y\}$ are linearly independent.
Hence $H$ contains a Hopf subalgebra isomorphic to $\cA_2$. This
proves the claim.

\smallbreak Now, since $\cA_2\simeq(\cA_2)^*$, applying \ref{H
contiene cA2} to $H^*$, we have that there exists
$\Pi:H\rightarrow\cA_2$ an epimorphism of Hopf algebras. Denote by
$\Pi\cE$ the subcoalgebra of $\cA_2$ generated by the
$\Pi(e_{ij})$'s. We will get a contradiction by trying to find the
dimension of $\Pi\cE$. Since $(\cA_2)_0\simeq C_2$, $\dim\Pi\cE<4$
and since $\Pi$ is an epimorphism, $\dim\Pi\cE\neq 0$. Next we
apply \cite[Thm. 2.1]{bitidasca} to $\Pi\cE$: if $\dim\Pi\cE\leq2$
then $\Pi\cE=\pi(C)\subseteq G(\cA_2)$ by \cite[Thm.
2.1]{bitidasca}, and hence $\pi(H)\subseteq k[G(\cA_2)]$; a
contradiction. If $\dim\Pi\cE=3$, by \cite[Thm. 2.1]{bitidasca},
$\Pi\cE$ is the linear span of two group-likes and a
skew-primitive. Thus $\Pi\cE=\Pi(C)$ is contained in a Hopf
subalgebra of $\cA_2$ isomorphic to $T_{-1}$. Then
$\Pi(H)\varsubsetneq\cA_2$, a contradiction.

Summarizing, $H^*$ cannot have a Hopf subalgebra isomorphic to
$\cA_2$. Then $H^*$ indeed is pointed.
\end{proof}

\end{document}